\documentclass[12pt]{article}
\input{amssym}
\newtheorem{obs}{Observation}[section]
\newtheorem{thm}[obs]{Theorem}

\newtheorem{cor}[obs]{Corollary}
\newtheorem{lemma}[obs]{Lemma}

\newtheorem{pro}[obs]{Proposition}

\newtheorem{preproof}{{\bf Proof.}}

\newenvironment{proof}[1]{\begin{preproof}{\rm
               #1}\hfill{$\rule{2mm}{2mm}$}}{\end{preproof}}

\def\newpic#1{}
\date{}
\begin{document}
\title{
{\Large{\bf On doubly resolving sets  in graphs }}}
%
{\small
\author{Mohsen Jannesari
\\
[1mm]
{\small \it  Department of  Science}\\
{\small \it  Shahreza Campus, University of Isfahan, Iran} \\
{\small \it Email: m.jannesari@shr.ui.ac.ir}
}}
\maketitle \baselineskip15truept
\begin{abstract}
Two vertices $u,v$ in a connected graph $G$ are doubly resolved by vertices $x,y$ of $G$ if
$$d(v,x)-d(u,x)\neq d(v,y)-d(u,y).$$
 A set $W$ of vertices of the graph $G$ is
a doubly resolving set for $G$ if every two distinct vertices of $G$ are doubly resolved by some two vertices of $W$.
Doubly resolving number of a graph $G$, denoted by
$\psi(G)$, is the minimum cardinality of a doubly resolving set for the graph $G$. The aim of this paper is to investigate
doubly resolving sets in graphs. An upper bound for $\psi(G)$ is obtained in terms of order and diameter of $G$.
$\psi(G)$ is computed for some  graphs and all graphs $G$ of order $n$ with the property $\psi(G)=n-1$ are determined.
Also, doubly resolving sets for unicyclic graphs are studied and it is proved that the difference between
the number of leaves and  doubly resolving number of a unicyclic graph is at most $2$.
\end{abstract}
{\bf Keywords:} doubly Resolving set;   doubly resolving number; resolving set; unicyclic graph.
\\
\par {\bf AMS Mathematical Subject Classification [2020]:} 05C12
\section{Introduction}
In this section, we present some definitions and known results
which are necessary  to prove our main results. Throughout this
paper, $G$ is a  simple connected graph
with vertex set $V(G)$, edge set $E(G)$ and order $n(G)$.
The distance between two
vertices $u$ and $v$, denoted by $d_G(u,v)$, is the length of a
shortest path between $u$ and $v$ in $G$. Also, $N_G(v)$ is the
set of all neighbors of vertex $v$ in $G$. We write these simply
$d(u,v)$ and $N(v)$,  when no confusion can arise.  The diameter of a graph $G$ is
${\rm diam}(G)=\max_{_{u,v\in V(G)}} d(u,v)$. A unicyclic graph is a graph with exactly one cycle.
The
symbols $(v_1,v_2,\ldots, v_n)$ and $(v_1,v_2,\ldots,v_n,v_1)$
represent a path of order $n$, $P_n$, and a cycle of order $n$,
$C_n$, respectively.
\par
 For an ordered subset $W=\{w_1,\ldots,w_k\}$ of $V(G)$ and a
vertex $v$ of $G$, the {\it metric representation}  of $v$ with
respect to $W$ is
$$r(v|W)=(d(v,w_1),\ldots,d(v,w_k)).$$
The set $W$ is  a {\it resolving set} for
$G$ if the distinct vertices of $G$ have different metric representations, with
respect to $W$.
A resolving set $W$ for $G$ with
minimum cardinality is  a {\it metric basis} of $G$, and its
cardinality is the {\it metric dimension} of $G$, denoted by
$\dim(G)$.
\par The concepts of resolving
sets and metric
dimension of a graph
 were introduced independently by Slater~\cite{Slater1975}
and by Harary and Melter~\cite{Harary}.
 Resolving sets have several applications in diverse areas such as  coin weighing problems~\cite{coin}, network
discovery and verification~\cite{net2}, robot
navigation~\cite{landmarks}, mastermind game~\cite{cartesian
product}, problems of pattern recognition and image
processing~\cite{digital}, and combinatorial search and
optimization~\cite{coin}.
For more results about  resolving sets and metric dimension see~\cite{baily,K dimensional,cartesian product,Ollerman,extermal}.
\par During the study
of the metric dimension of the cartesian product of graphs Caceres et al.~\cite{cartesian product} defined the concept
of doubly resolving sets in graphs.
 Two vertices $u,v$ in a  graph $G$ are doubly resolved by $x,y\in V(G)$ if
$$d(v,x)-d(u,x)\neq d(v,y)-d(u,y).$$
 A set $W$ of vertices of the graph $G$ is
a doubly resolving set for $G$ if every two distinct vertices of $G$ are doubly resolved by some two vertices of $W$.
Every graph with at least two vertices has a doubly resolving set.
A doubly resolving set for $G$ with minimum cardinality is called a doubly basis of $G$ and
 its cardinality is called the doubly resolving number of $G$ and
denoted by
$\psi(G)$.
\par Caceres et al.~\cite{cartesian product} obtained  doubly resolving number of
trees, cycles and complete graphs.
In~\cite{computing minimal doubl} it was proved that the problem of finding doubly bases is NP-hard.
Doubly resolving number of Prism graphs and Hamming graphs is computed in \cite{prism} and  \cite{Hamming}, respectively.
For more results about doubly resolving sets in graphs see\cite{necklace,cartesian product,computing minimal doubl,cocktail}.
\par Caceres et al.~\cite{cartesian product} obtained an upper bound for the metric dimension of cartesian product of graphs $G$ and $H$
in terms of $\psi(H)$ and $\dim(G)$.
They also obtained a lower bound for the metric dimension of the cartesian product of a graph $G$ with itself in terms of $\psi(G)$.
Hence computing doubly resolving number of graphs is useful for computing  metric dimension of
cartesian product of graphs. Moreover, studying doubly resolving sets is interesting by itself.
It is clear that for each graph of order at least $2$, we have $\psi(G)\geq2$. Caseres et al. found the following upper bound for $\psi(G)$.
\begin{lemma}\label{y(G)<n-1}{\rm\cite{cartesian product}}
For every graph $G$ with $n\geq3$ vertices we have $\psi(G)\leq n-1$.
\end{lemma}
Also, through the  following three lemmas,
they  found the  doubly resolving number  of complete graphs, paths and cycles.
\begin{lemma}\label{y(Kn)=n-1}{\rm\cite{cartesian product}}
For all $n\geq2$ we have $\psi(K_n)=\max\{n-1,2\}$.
\end{lemma}
\begin{lemma}\label{Y(Pn)=2}{\rm\cite{cartesian product}}
For each $n\geq2$, $\psi(P_n)=2$.
\end{lemma}
\begin{lemma}\label{Y(Cn)}{\rm\cite{cartesian product}}
Let $C_n$ be a cycle of order $n$. Then
$$\psi(C_n)=\left\{
\begin{array}{ll}
2 &  {\rm if~}n~{\rm is~odd}, \\
3 &  {\rm if~}n~{\rm is~even}.
\end{array}\right.$$
\end{lemma}
\par In this paper we investigate  doubly resolving sets for graphs. In Section~\ref{results},
 some properties of doubly resolving bases of graphs are presented,
an upper bound for $\psi(G)$ in terms of diameter and order of $G$ is obtained and
the doubly resolving number of complete bipartite graphs  is computed. In Section~\ref{n-1},
all $n$-vertex graphs $G$ with $\psi(G)=n-1$ are determined. In section~\ref{unicyclic}, doubly resolving sets of unicyclic graphs
are investigated and it is proved  that the difference between
the number of leaves and  doubly resolving number of a unicyclic graph is at most $2$.
\section{Some results on doubly resolving sets}\label{results}
In this section,  we present some results about doubly resolving sets of graphs. We obtain the doubly resolving number of some famous
families of graphs. An important upper bound for doubly resolving number of graphs is obtained in terms of order and diameter of the graph.
\par
By the following lemma, to check whether a given set $W\subseteq V(G)$ is a doubly resolving set, it does not need to consider the pair
of vertices that both of them are in $W$.
\begin{lemma}\label{chek W doubly resolving}
Let $W$ be a subset of size at least $2$ of $V(G)$. If $x,y\in W$, then $x,y$ are doubly resolved by  $x,y\in W$.
Also if $a\in W$ and $b\notin W$ are two vertices in $G$ that are not doubly resolved by any pair of vertices in $W$, then
for each $w\in W$
there exists a shortest path between $b$ and $w$ that contains $a$.
\end{lemma}
\begin{proof}{
Let $x,y\in W$. Then $d(x,x)-d(y,x)=-d(x,y)\neq d(x,y)=d(x,y)-d(y,y).$ That is $x,y$ are doubly resolved by $x,y$.
Now let $a\in W$, $b\in V(G)\setminus W$ and $a,b$ are not doubly resolved by any pair of vertices in $W$. For each $w\in W$,
we have
$d(a,w)-d(b,w)=d(a,a)-d(b,a)$, which implies $d(b,w)=d(b,a)+d(a,w)$. If $P$ and $Q$ are shortest paths
between $b,a$ and $a,w$, respectively, then $P\cup Q$ is a path between $b,w$ with length $d(b,w)$.
Therefore $P\cup Q$ is a shortest path between $b,w$ that contains $a$.
}\end{proof}

 Two distinct vertices $u,v$ are said to be {\it twins}
if $N(v)\backslash\{u\}=N(u)\backslash\{v\}$. It is
clear that if $u,v$ are twin vertices in $G$, then for every vertex $x\in V(G)$, $d(u,x)=d(v,x)$.
\begin{pro}\label{twins}
Suppose that $u,v$ are twins in a  graph $G$ and $W$ is a doubly resolving set for $G$.
Then at least on of the vertices $u$ and $v$ is in $W$. Moreover, if $u\in W$ and $v\notin W$,
then $(W\setminus\{u\})\cup \{v\}$ is also a doubly resolving set for $G$.
\end{pro}
\begin{proof}{
Let $u,v$ be twins, $W$ be a doubly resolving set and $u,v\notin W$. Then for each $a,b\in W$,
$$d(u,a)-d(v,a)=0=d(u,b)-d(v,b).$$ This contradiction implies that $W$ must contain at least one of $u$ or $v$.
Now let $W$ be a doubly resolving set for $G$ such that  $u\in W$ and $v\notin W$. If
$(W\setminus\{u\})\cup \{v\}$ is not a doubly resolving set for $G$, then there are vertices $x,y\in V(G)$
and $w\in W$ such
that $x,y$ are doubly resolved by $u,w$ and are not doubly resolved by $v,w$. But
$$ d(x,v)-d(y,v)=d(x,u)-d(y,u)\neq d(x,w)-d(y,w).$$
This means $(W\setminus\{u\})\cup \{v\}$ is  a doubly resolving set for $G$.
}\end{proof}
One of the well-known families of graphs is the family of complete bipartite graphs. By the next lemma
we compute the doubly resolving number of complete bipartite graphs.
\begin{lemma}\label{Y(K r,s)}
Let $K_{r,s}$ be a complete bipartite graph of order $n\geq3$ and $r\leq s$. Then
$$\psi(K_{r,s})=\left\{
\begin{array}{ll}
n-1 & ~~~~ {\rm if~}r\leq 2, \\
n-2 & ~~~~ {\rm if~}r>2.
\end{array}\right.$$
\end{lemma}
\begin{proof}{
Suppose that $K_{r,s}$ be a complete bipartite graph with partite sets $X,Y$ such that $|X|=r$ and $|Y|=s$. Since $n\geq3$
and $r\leq  s$, we have $s\geq2$. Let $W$ be a doubly
resolving set for $K_{r,s}$. One of the following three cases can be arisen.
\par Case 1. $r=1$. If $s=2$, then by Lemma~\ref{Y(Pn)=2},  $\psi (K_{1,2})=2=n-1$. Let $s\geq3$. Then
  all vertices in $Y$ are twins and by Proposition~\ref{twins}, $|W\cap Y|\geq s-1$.
 Hence $|W|\geq n-2$. If $|W|= n-2$, then there exists a vertex
 $y_0\in Y\setminus W$. Now, for each $w\in W$,
 $d(y_0,w)-d(x_0,w)=1$, where $X=\{x_0\}$. That means $W$ is not a doubly resolving set.
 This contradiction yields $\psi(K_{1,s})=s=n-1$.
 \par Case 2. $r=2$. In this case all vertices in $X$ are twins, also all vertices in $Y$ are twins.
  Hence by Proposition~\ref{twins}, $|W\cap Y|\geq s-1$ and $|W\cap X|\geq 1$. Therefore
  $|W|\geq n-2$. If $|W|= n-2$, then $ X\cap W=\{x_1\}$ and there exists a vertex
 $y_0\in Y\setminus W$. Note that for each $y\in W\cap Y$, we have $d(y_0,y)-d(x_1,y)=2-1=1$
 and  $d(y_0,x_1)-d(x_1,x_1)=1-0=1$, which is impossible. Therefore, $\psi(K_{2,s})=s+1=n-1$.
 \par Case 3. $r\geq3$. In this case all vertices in $X$ are twins, also all vertices in $Y$ are twins.
  Hence by Proposition~\ref{twins}, $|W\cap Y|\geq s-1$ and $|W\cap X|\geq r-1$. Therefore
  $|W|\geq n-2$. Let $W$ be a set of vertices of size $n-2$ such that $|W\cap Y|= s-1$ and $|W\cap X|= r-1$.
  Consider $a,b\in V(K_{r,s})$. If $a,b\in Y$, then  by Lemma~\ref{chek W doubly resolving} we can assume
  that $a\in W$ and $b\notin W$. Since $s\geq3$, there exists a vertex $a\neq w\in W\cap Y$ and we have
  $d(a,a)-d(b,a)=-2$ and $d(a,w)-d(b,w)=0$. Hance $a,b$ are doubly resolved by $a,w$. If $a,b\in X$, by a same argument we
  deduce that $a,b$ are doubly resolved by a pair of vertices in $W$. Now let $a\in X$ and $b\in Y$. Since $r,s$ are at least $3$, there exists
  vertices $x\in W\cap X\setminus\{a\}$ and $y\in W\cap Y\setminus\{b\}$. Note that
  $$d(a,x)-d(b,x)=2-1\neq 1-2=d(a,y)-d(b,y).$$
  Therefore $W$ is a doubly resolving set for $K_{r,s}$ and $\psi(K_{r,s})=n-2$.
}\end{proof}
Through the following theorem we obtain an upper bound for $\psi(G)$ in terms of order and diameter of the graph $G$.
\begin{thm}\label{Y(G)<n-d+1}
If $G$ is a graph with diameter $d$, then $\psi(G)\leq n-d+1$.
\end{thm}
\begin{proof}{
Let $P=(v_0,v_1,\ldots,v_d)$ be a shortest path in $G$. We claim that $W=V(G)\setminus\{v_1,v_2,\ldots,v_{d-1}\}$
is a doubly resolving set for $G$.
If $v_i,v_j\in V(G)\setminus W$, then
$$d(v_i,v_0)-d(v_j,v_0)=i-j\neq j-i=(d-i)-(d-j)=d(v_i,v_d)-d(v_j,v_d).$$ Therefore $v_0$ and $v_d$ doubly resolve $v_i$ and $v_j$.
If $v_i\in V(G)\setminus W$ and $a\in W$ are two vertices that are not doubly resolved by $W$, then
$$d(a,a)-d(v_i,a)=d(a,v_0)-d(v_i,v_0) \Rightarrow d(v_0,a)=d(v_i,v_0)-d(v_i,a),$$
Since $a,v_i$ are not doubly resolved by $a,v_0$. Also
$$d(a,a)-d(v_i,a)=d(a,v_d)-d(v_i,v_d) \Rightarrow d(a,v_d)=d(v_i,v_d)-d(v_i,a),$$
Since $a,v_i$ are not doubly resolved by $a,v_d$. Hence,
$$ d(v_0,v_d)\leq d(v_0,a)+d(a,v_d)=d(v_i,v_0)-d(v_i,a)+d(v_i,v_d)-d(v_i,a)=d(v_0,v_d)-2d(v_i,a).$$
This contradiction implies that $W$ is a doubly resolving set for $G$ and $\psi(G)\leq |W|=n-d+1$.
}\end{proof}
By Lemma~\ref{Y(Pn)=2} for every $n\geq2$, $\psi(P_n)=2=n-\rm diam(P_n)+1$. Therefore the bound in Theorem~\ref{Y(G)<n-d+1}
is sharp.
\section{Graphs of order $n$ and doubly resolving number $n-1$}\label{n-1}
The aim of this section is to determine all $n$-vertex
graphs $G$ with $\psi(G)=n-1$. To do this we need to compute doubly resolving number of some graphs.
\par
Let $G$ and $H$ be two graphs with disjoint vertex sets. The {\it join} of  $G$ and $H$, denoted by $G\vee H$, is the
graph with vertex set  $V(G)\cup V(H)$ and edge set $E(G)\cup
E(H)\cup\{uv|\, u\in V(G),v\in V(H)\}$. To find all graphs $G$ of order $n$ with the property $\psi(G)=n-1$, we need to
compute $\psi(k_2\vee\overline{K_n}).$

\begin{lemma}\label{Y(k2 v Ks bar}
If $\overline{K_n}$ is the complement graph of $K_n$, then $\psi(K_2\vee\overline{K_n})=n+1.$
\end{lemma}
\begin{proof}{
Let $K_2\vee\overline{K_n}=G$, $X=V(K_2)$ and $Y=V(K_n)$.
 If $n=1$, then $G=p_3$ and $\psi(G)=2=n+1$. Now consider $n\geq 2$. Clearly all vertices in $X$ are twins, also all vertices in $Y$ are twins.
Let $W$ be a doubly resolving set for $G$. By Proposition~\ref{twins}, $|W\cap X|\geq1$ and $|W\cap Y|\geq n-1$. Thus $|W|\geq n$.
If $|W|=n$,  then $ X\cap W=\{x_1\}$ and there exists a vertex
 $y_0\in Y\setminus W$. Note that for each $y\in W\cap Y$, we have $d(y_0,y)-d(x_1,y)=2-1=1$
 and  $d(y_0,x_1)-d(x_1,x_1)=1-0=1$. That is $y_0,x_1$ are not doubly resolved by any pair
 of vertices in $W$, a contradiction. Therefore, $\psi(G)=|W|=n+1$.
}\end{proof}

 First, we find all graphs $G$ of order $n$, maximum degree $n-1$ and
$\psi(G)=n-1$.
\begin{pro}\label{Delta=n-1=psi}
Let $G$ be a graph of order $n\geq3$ and maximum degree $\Delta$.
If $\psi(G)=\Delta=n-1$, then $G$ is $K_n, K_{1,n-1},$ or $K_2\vee \overline{K_{n-2}}$.
\end{pro}
\begin{proof}{
 By Theorem~\ref{Y(G)<n-d+1},
 it follows that $\rm diam(G)\leq2$.
   Suppose that $x$ is a vertex of degree $n-1$ and
  $I$ is a maximum independent subset of $N(x)$. If $I=N(x)$, then $G=K_{1,n-1}$ and if $|I|=1$, then $G=K_n$ .
Hence assume that $2\leq |I|\leq n-2$. This implies that there exists a vertex $y\in N(x)\setminus I$.
We claim that $y$ is adjacent ro all vertices in $I$.
 By definition of $I$, $y$ has a neighbour
in $I$, say $a$.
 Since $\psi(G)=n-1$, the set $W=V(G)\setminus \{x,y\}$ is not a doubly resolving set for $G$.
If there exists a vertex $b\in I$ that is not adjacent to $y$, then
$$d(x,a)-d(y,a)=1-1\neq 1-2=d(x,b)-d(y,b).$$
This means $x,y$ are doubly resolved by $a,b\in W$.
Note that $2\le|I|\leq|W|$, thus for each   $v\in W$,  there exists  $v\neq u\in  W$ and we have
$$d(x,u)-d(v,u)=1-d(v,u)\neq 1=d(x,v)=d(x,v)-d(v,v).$$
Therefore for every $v\in W$,  $x,v$ are doubly resolved by some pair of vertices in $W$. Since $W$ is not a doubly resolving set,
Lemma~\ref{chek W doubly resolving} implies that there exists a vertex $w\in W$ such that $y,w$ are not doubly
resolved by any pair of vertices in $W$. Hence,
$$d(y,w)=d(y,w)-d(w,w)=d(y,a)-d(w,a)=1-d(w,a).$$ But $d(y,w)\geq1$ and $d(w,a)\geq0$ imply
$d(w,a)=0$ and $w=a$. Note that $\rm diam(G)\leq2$, yields
$$d(a,b)-d(y,b)=2-2\neq 0-1=d(a,a)-d(y,a).$$
This contradiction means that $y$ is adjacent to $b$. Therefore $y$ is adjacent to all vertices in $I$.
Since $y$ is an arbitrary vertex in $N(x)\setminus I$, the above argument implies that all vertices
in $I$ are adjacent to all vertices in $N(x)\setminus I$. We claim that $N(x)\setminus I=\{y\}$. Suppose on the contrary
that  $|N(x)\setminus I|\geq2$. Let $W=V(G)\setminus \{a,x\}$ and $b\in I\cap W$.
We have
$$d(a,b)-d(x,b)=2-1\neq 1-1=d(a,y)-d(x,y),$$
and so $a,x$ are doubly resolved by $b,y\in W$. Also
 $$d(x,y)-d(y,y)=1\neq 0=d(x,b)-d(y,b)$$
  and for each
 $y\neq w\in W$ we have
 $$d(x,y)-d(w,y)\neq 1=d(x,w)-d(w,w).$$
 Therefore for every $w\in W$, vertices $x,w$ are doubly resolved by some vertices in $W$.
 Since $W$ is not a doubly resolving set, there exists a vertex $w\in W$ such that $w,a$ are not doubly resolved by
 any pair of vertices in $W$. If $w\in I\cap W$, then
 $$d(w,y)-d(a,y)=1-1\neq -2=d(w,w)-d(a,w),$$
 which is impossible. Thus $w\in N(x)\setminus I$.  Let $w\neq y'\in N(x)\setminus I$, this implies
 $d(a,y')-d(w,y')=1-d(w,y')\leq 0$  and $d(a,w)-d(w,w)=1$. That means $a,w$ are doubly resolved by $w,y'\in W$.
 This contradiction leads us to $N(x)\setminus I=\{y\}$. Therefore $|I|=n-2$ and $G=K_2\vee\overline{K_{2,n-2}}$.
}\end{proof}
The next theorem determines all graphs $G$ of order $n$ and $\psi(G)=n-1$.
\begin{thm}\label{CHR Y(G)=n-1}
Let $G$ be a graph of order $n\geq3$. Then $\psi(G)=n-1$ if and only if $G$ is $K_n, K_{1,n-1},K_{2,n-2},$ or $K_2\vee \overline{K_{n-2}}$.
\end{thm}
\begin{proof}{
By Lemmas~\ref{y(Kn)=n-1},~\ref{Y(K r,s)} and \ref{Y(k2 v Ks bar}  the doubly resolving number of each of the
 graphs mentioned in the statement of the theorem is $n-1$.
 \par For the converse, assume that $G$ is a graph of order $n\geq3$ such that $\psi(G)=n-1$. By Theorem~\ref{Y(G)<n-d+1},
 it follows that $\rm diam(G)\leq2$. Let $\Delta$ be the maximum degree in $G$. Since $G$ is
  connected and has at least $3$ vertices, we have $\Delta\geq2$.
  \par If $\Delta=2$, then $G$ is a path $P_n$ or a cycle $C_n$. If
  $G=P_n$, then by Lemma~\ref{Y(Pn)=2},
   $n=3$ and $G=P_3=K_{1,2}$.  In the case $G=C_n$, Lemma~\ref{Y(Cn)} implies that $n\in \{3,4\}$ and $G=C_3=K_3$ or
  $G=C_4=K_{2,2}$.
  \par Now let $\Delta\geq3$. If $\Delta=n-1$, then by Proposition~\ref{Delta=n-1=psi},
   $G$ is $K_n, K_{1,n-1},$ or $K_2\vee \overline{K_{n-2}}$.
 If $3\leq \Delta\leq n-2$, then
 let $x$ be a vertex of degree $\Delta$ and $y$ be a non-adjacent vertex to $x$. Since $\rm diam(G)\leq 2$, vertices
 $x,y$ have a common neighbour, say $a$. $\Delta\geq3$ implies $x$ has at least two more neighbours, say $b$ and $c$.
 Clearly, $W=V(G)\setminus \{a,x\}$ is not a doubly resolving set for $G$. But $a,x$ are doubly resolved by $y,b$ because
 $$d(a,y)-d(x,y)=1-2=-1\neq d(a,b)-d(x,b)=d(a,b)-1\geq0.$$
 Also for each $w\in W$, vertices $x,w$ are doubly resolved by some pair of vertices in $W$. To see
 this we must consider two cases $w=c$ or $w\neq c$. If $w=c$, then $d(c,c)-d(x,c)=-d(x,c)=-1$ and
 $d(c,b)-d(x,b)=d(c,b)-1\geq0$. If $w\neq c$, then $d(w,w)-d(x,w)=-d(x,w)<0$ and $d(w,c)-d(x,c)=d(w,c)-1\geq0$.
 Since $W$ is not a doubly resolving set for $G$, there exists a vertex $w\in W$ such that $a,w$ are not
 doubly resolved by any pair of vertices in $W$. Therefore
 $$1\leq d(a,w)=d(a,w)-d(w,w)=d(a,y)-d(w,y)=1-d(w,y).$$
 This means $w=y$. Note that
 $1=d(a,y)-d(y,y)=d(a,c)-d(y,c).$ Hence $d(a,c)=2$ and $d(y,c)=1$. Since $c$ is an arbitrary neighbour of $x$, we deduce that
 $N(y)=N(x)$ and $N(x)$ is an independent set. We claim that $V(G)=N(x)\cup\{x,y\}$. Suppose on the contrary that
 there exists a vertex $ y'\in V(G)\setminus (N(x)\cup\{x,y\})$. By
 Similar to the above argument we have
  $N(y')=N(x)$.  Since $\psi(G)=n-1$, the set  $W=V(G)\setminus\{y',c\}$ can not be a doubly
 resolving set for $G$. But $y',c$ are doubly resolved by $y,b$. If $w\in W\cap N(x)$, then
 $$d(c,w)-d(w,w)=2\neq 0=d(c,x)-d(w,x).$$
For $w\in W\setminus N(x)$, let $w\neq w'\in W\setminus N(x)$. Hence
 $$d(c,w')-d(w,w')=1-2\neq 1=d(c,w)-d(w,w).$$
 Therefore for every $w\in W$, vertices $c,w$ are doubly resolved by some pair of vertices in $W$.
 By a same way for every $w\in W$, vertices $y',w$ are doubly resolved by some pair of vertices in $W$.
 This means $W$ is a doubly resolving set for $G$, which is a contradiction. Thus $V(G)=N(x)\cup\{x,y\}$, $N(x)=N(y)$ is an independent set,
 and $x$ is not adjacent to $y$. Therefore $G=K_{2,n-2}$.
}\end{proof}
\section{Doubly resolving number of unicyclic graphs}\label{unicyclic}
In this section we investigate doubly resolving sets in uncyclic graph. A unicyclic graph  is a graph that have exactly
one cycle. If $G$ is a unicyclic graph $C(G)$ indicates the unique cycle of $G$.
For each vertex $x\in V(C(G))$ of degree at least $3$ we define
$$V(x)=\{v\in V(G)\setminus V(C(G))|\forall y\in V(C(G)), d(v,x)< d(v,y)\},$$
and $T(x)$ as the induced subgraph $\langle \{x\}\cup V(x)\rangle$ of $G$. Clearly $T(x)$ is a tree.
 A  leaf
in a graph is a vertex of degree $1$. We use the notations $L(G)$ and $l(G)$ for the set of all leaves
 in the graph $G$ and its cardinality, respectively.
  Caceres et al.~\cite{cartesian product} proved the following lemma for doubly bases of trees.
\begin{lemma}\rm\cite{cartesian product}\label{bases in trees}
The set of leaves is the unique doubly resolving basis for a tree.
\end{lemma}
By the same proof we can see that each resolving set of a graph contains all leaves.
 The following lemma prepares a lower bound for doubly resolving number
in terms of the number of leaves in a graph.
\begin{lemma}\label{leaves}
Let $v$ be a vertex of degree $1$ in a graph $G$. Then $v$ belongs to all doubly basis of $G$, and
 $$\psi(G)\geq l(G).$$
\end{lemma}
\begin{proof}{
Let $B$ be a doubly basis of $G$ and $u$ be  the neighbour of $v$. If $v\notin B$,
 then for each $x\in B$, $d(v,x)-d(u,x)=1$ which is impossible. Therefore $v\in B$ and $\psi(G)\geq l(G).$
}\end{proof}
\begin{pro}\label{leaves are doubly resolving set}
Let $G$ be a unicyclic graph and $W$ be a doubly resolving set for $C(G)$. If every vertex of $W$ is of degree at
least $3$, then  $L(G)$ is a doubly basis of $G$.
\end{pro}
\begin{proof}{
Let $r,s\in V(G)$, we need to consider the following three cases.
\par Case $1$. $r,s\in V(C(G))$. Since $W$ is a doubly resolving set for $C(G)$, there are vertices $x,y\in W$ that
$r,s$ are doubly resolved by $x,y$. Let $x_1\in V(x)$ and $y_1\in V(y)$ be leaves. Then
\begin{eqnarray*}
d(r,x_1)-d(s,x_1)&= & d(r,x)+d(x,x_1)-(d(s,x)+d(x,x_1))=\\
&=& d(r,x)-d(s,x)\neq d(r,y)-d(s,y)=\\
&=&  d(r,y)+d(y,y_1)-(d(s,y)+d(y,y_1))=\\
&= & d(r,y_1)-d(s,y_1).
\end{eqnarray*}
Therefore $r,s$ are doubly resolved by leaves $x_1,y_1$.
\par Case $2$. $r\in V(C(G))$ and $s\notin V(C(G))$. Let $s\in V(t)$ for some $t\in V(C(G))$. Hence
there are vertices $x,y\in W$ such that $r,t$ are doubly resolved by $x,y$. There are two possibilities, $t\in\{x,y\}$ or $t\notin\{x,y\}$.
If $t\in\{x,y\}$, say $t=x$, then let $y_1$ be a leaf in $V(y)$ and $x_1$ be a leaf in $V(x)$ such that $s$ is a vertex in the shortest path
between $x$ and $x_1$. Hence $d(x,x_1)=d(x,s)+d(s,x_1)$. If $r,s$ are not resolved by $x_1,y_1$, then
\begin{eqnarray*}
d(r,x)+d(x,s)&= &(d(r,x)+d(x,s)+d(s,x_1))-d(s,x_1)=d(r,x_1)-d(s,x_1)=\\
&=& d(r,y_1)-d(s,y_1)=d(r,y)+d(y,y_1)-(d(s,y)+d(y,y_1))=\\
&=&  d(r,y)-(d(s,y)=d(r,y)-(d(s,x)+d(x,y)).
\end{eqnarray*}
Hence $d(r,x)+d(x,y)=d(r,y)-2d(x,s)$.
But  we know that $d(r,y)\leq d(r,x)+d(x,y)$ satisfies for all three vertices $x,y,r\in V(G)$. Therefore
$$d(r,y)\leq d(r,x)+d(x,y)=d(r,y)-2d(s,x)\leq d(r,y)-2,$$
which is impossible. Thus $r,s$ are doubly resolved by $x_1,y_1$. If $t\notin\{x,y\}$, let $x_1$  be a leaf in $V(x)$ and
$y_1$ be a leaf in $V(y)$. Then
\begin{eqnarray*}
d(r,x_1)-d(s,x_1)&= & d(r,x)+d(x,x_1)-(d(s,t)+d(t,x)+d(x,x_1))=\\
&=& d(r,x)-(d(s,t)+d(t,x))\neq d(r,y)-(d(s,t)+d(t,y))=\\
&=&  d(r,y)+d(y,y_1)-(d(s,t)+d(t,y)+d(y,y_1))=\\
&= & d(r,y_1)-d(s,y_1).
\end{eqnarray*}
Therefore $r,s$ are doubly resolved by leaves $x_1,y_1$.
\par Case $3$. $r,s\notin V(C(G))$. Let $r\in V(y)$ and $s\in V(x)$ for some $x,y\in V(C(G))$.
If $x\neq y$, let $x_1\in V(x)$ be a leaf such that $s$ is in the shortest path between $x$ and $x_1$
and $y_1\in V(y)$ be a leaf such that $r$ is in the shortest path between $y$ and $y_1$. Thus
$$d(r,x_1)-d(s,x_1)=d(r,y)+d(y,x)+d(x,s)+d(s,x_1)-d(s,x_1)>0,$$
and
$$d(r,y_1)-d(s,y_1)=d(r,y_1)-(d(s,x)+d(x,y)+d(y,r)+d(r,y_1))<0.$$
Therefore $r,s$ are doubly resolved by $x_1,y_1$. If $x=y$, then by Lemma~\ref{bases in trees} leaves of
$T(x)$ is a basis for $T(x)$. Let $u,t$ be two leaves of $T(x)$ that doubly resolve $r,s$. If $x\notin\{u,t\}$, then
$u,t$ are leaves of $G$. If $x\in\{u,t\}$, say $x=t$, then let $y_1\in V(G)\setminus V(x)$ be a
leaf of $G$. Thus
$$d(r,y_1)-d(s,y_1)=d(r,x)+d(x,y_1)-(d(s,x)+d(x,y_1))=d(r,x)-d(s,x)\neq d(r,t)-d(s,t).$$
Therefore $r,s$ are doubly resolved by $x,t$  and the result follows.
}\end{proof}
\begin{cor}\label{L(v)uWis resolving}
Let $G$ be a unicyclic graph and $W$ be a doubly resolving set for $C(G)$. If $U$ is the set of all vertices
of degree $2$ in $W$,
 then  $L(G)\cup U$ is a doubly resolving set for $G$.
\end{cor}
\begin{proof}{
Let $U=\{u_1,u_2,\ldots,u_t\}$. We construct the graph $G'$ by adding  leaves $v_i;1\leq i\leq t$,
 to the graph $G$ such that  for each $i;1\leq i\leq t$, $v_i$ is adjacent to $u_i$.
 By Proposition~\ref{leaves are doubly resolving set},
  $L(G')$ is a doubly basis of $G'$.  Note that $L(G')=L(G)\cup\{v_1,v_2\ldots,v_t\}$. Let $x,y$ be two vertices in $G$. Then
  for every $i$, $d(x,v_i)-d(y,v_i)=d(x,u_i)-d(y,u_i)$. Therefore $L(G)\cup U$ is a doubly resolving set for $G$.
}\end{proof}

The next  theorem obtains an upper bound for the doubly resolving number of unicyclic graphs.
\begin{thm}\label{doubly of unicyclic}
Let $G$ be a unicyclic graph that is not a cycle. If $C(G)$ is of order $m$, then
$$ l(G)\leq \psi(G)\leq\left\{
\begin{array}{ll}
l(G)+1 & ~~~~ {\rm if~}m{\rm ~is~ odd}, \\
l(G)+2 & ~~~~ {\rm if~}m{\rm ~is ~ even}.
\end{array}\right.$$
\begin{proof}{
Let $C(G)=(v_1,v_2,\ldots,v_m,v_1)$.  If $m$ is odd
by lemma~\ref{Y(Cn)}, $\psi(C(G))=2$. Let $\{v_i,v_j\}$ be a doubly basis of $C(G)$. Since $G$ is not a cycle at least one of the vertices in
$C(G)$ is of degree $3$. By relabeling vertices of $C(G)$,  we can assume that $v_i$ is of degree at least $3$.
 Hence by Corollary~\ref{L(v)uWis resolving},
 $W=L(G)\cup\{v_j\}$ is a doubly resolving set for $G$. Therefore $\psi(G)\leq l(G)+1$. If $m$ is even
by lemma~\ref{Y(Cn)}, $\psi(C(G))=3$. The same argument implies that $\psi(G)\leq l(G)+2$.
}\end{proof}
\end{thm}
\begin{cor}\label{Y(G)=l(G) cor}
Let $G$ be a unicyclic graph that is not a cycle and $B$ be a doubly basis of $C(G)$. If all vertices of $B$ are of degree at least $3$,
then $\psi(G)=l(G)$.
\end{cor}

\end{document}